\documentclass{amsart}

\newtheorem{theorem}{Theorem}[section]

\theoremstyle{definition}
\newtheorem{definition}[theorem]{Definition}
\newtheorem{example}[theorem]{Example}

\newtheorem{proposition}[theorem]{Proposition}

\theoremstyle{remark}

\numberwithin{equation}{section}

\begin{document}

\title[Geometric properties of second Ricci solitons]{ Geometric properties of second Ricci solitons  }

\author{Masoumeh Khalili}
\address{ Department of Mathematics, Faculty of Sciences
Imam Khomeini International University,
Qazvin, Iran.}
\email{mkhalili@ikiu.ac.ir}
\author{Ghodratallah Fasihi-Ramandi}
\address{Department of Mathematics, Faculty of Sciences
Imam Khomeini International University,
Qazvin, Iran. }
\email{fasihi@sci.ikiu.ac.ir}
\author{Shahroud Azami}
\address{Department of Mathematics, Faculty of Sciences
Imam Khomeini International University,
Qazvin, Iran. }
\email{azami@sci.ikiu.ac.ir}
\subjclass[2010]{53C21, 53C25, 35C42}



\keywords{ Hyperbolic Ricci soliton, Second Ricci soliton, Evolutionary equation}
\begin{abstract}
This paper introduce the idea of second Ricci solitons. A second Ricci soliton is nothing but a steady hyperbolic Ricci soliton. We study the geometry of closed and compact second Ricci soliton manifolds.  Immersed submanifolds as second solitons also will be investigated. Finally, we investigate this structure on warped product manifolds.
\end{abstract}

\maketitle

\section{Introduction}

A novel concept known as the hyperbolic Ricci flow first unveiled by Liu and Kong \cite{kong}. This isn't just any partial differential equation; it's a profound evolution equation that governs how a Riemannian metric, denoted as $g$, transforms over time on a smooth manifold, $M$. Essentially, it dictates the wave dynamics of $g$, directly linking its second time derivative to the Ricci curvature, ${\rm Ric}$, and the metric itself
\begin{equation}\label{flow}
\frac{\partial^2g}{\partial t^2}(t) = -2{\rm Ric}_{g(t)},\quad g_0=g(0).
\end{equation}
This elegant expression reveals that the acceleration of the metric's change is directly proportional to its Ricci curvature and its current state at any given moment.\\

Self-similar solutions for this flow known as hyperbolic Ricci solitons first introduced by Fasihi and Azami \cite{me1, {faraji}}. In fact, we introduce a family of time-dependent vector fields, $W(t)$, which are tangent to the manifold $M$. These fields are constructed using an initial vector field, $W_0$, and a smooth, positive real-valued function, $f(t)$ where $f(0) = 1$ by $W(t) = \frac{1}{f(t)}W_0$.\\
Accompanying this, we define a group of diffeomorphisms, $\{\varphi_t\}$, which essentially map points on the manifold as the flow progresses, tracing the path dictated by the vector fields.

With these assumptions, we define a family of Riemannian metrics, $\{g(t)\}$, where $g(t) = f(t)\varphi_t^*(g_0)$, with $g_0$ being the metric at the initial time $t=0$. By differentiating this definition with respect to time, we can obtain the velocity and acceleration of the metric's transformation. These derivations inherently involve the first and second derivatives of the function $f(t)$, as well as the crucial Lie derivatives of the initial metric $g_0$ with respect to the vector field $W_0$. The Lie derivative, for those curious, quantifies how a tensor changes along the direction of a vector field.\\

By combining our initial evolution equation \eqref{flow} with the insights gleaned from our time derivations, and by evaluating these expressions at $t=0$, we arrive at a pivotal equation:

\[\mathcal{L}_{W_0}\mathcal{L}_{W_0}g_0 + f'(0)\mathcal{L}_{W_0}g_0 +{\rm Ric_{g_0}}= -f''(0)g_0,\]

here, ${\rm Ric}$ represents the Ricci curvature of our initial metric $g_0$. If we assign specific values to the initial rates of change of $f$ (let $f'(0) = \mu$ and $f''(0) = -\lambda$), we finally obtain the defining equation for what are known as hyperbolic Ricci solitons:
\begin{equation}\label{hrf}
\mathcal{L}_{W}\mathcal{L}_{W}g + \mu\mathcal{L}_{W}g +{\rm Ric}= \lambda g.
\end{equation}

This  equation describes the stationary solutions of the hyperbolic Ricci flow, essentially representing metrics that maintain their form under this specific evolutionary process. The vector field $W$ in this context is called potential vector field. We refer to such a structure as an $n$-dimensional hyperbolic soliton, characterized by the tuple $(M^n, g, W, \mu, \lambda)$, encompassing the $n$-dimensional manifold $M$, the metric $g$, the potential vector field $W$, and the constants $\mu$ and $\lambda$.

When $W$ is Killing, then the hyperbolic Ricci soliton \eqref{hrf} reduced to be an Einstein manifold. Also, if $W$ be 2-Killing (i.e., $\mathcal{L}_{W}\mathcal{L}_{W}g=0$) \cite{7} and $\mu=\frac{1}{2}$, then \eqref{hrf} becomes Ricci soliton.

Hyperbolic Ricci solitons are new and attracted many researcher's interest. For more details on hyperbolic Ricci solitons as well as the new researches in this topic one can consult \cite{sequen, {warp}, {blag}, {blaga}, {faraji}, {shamkhali1}, {pahan}, {md}, {Shamkhali3}}.

A hyperbolic soliton \eqref{hrf} is said to be expanding, shrinking or steady if $\mu>0, \mu<0$ or $\mu=0$, respectively. Clearly, when $\mu=f'(0)=0$ we reach to a critical situation, so it is important to stop here and focus on this case. Inspired by this fact, we present the following definition.
\begin{definition}
A semi-Reimannian manifold $(M^n,g)$ is said to be a second Ricci solitons, if and only if there exists a vector field $W\in\mathcal{X}(M)$ and a real scalar $\lambda$, such that
\begin{equation}\label{metal}
\mathcal{L}_{W}\mathcal{L}_{W}g+{\rm Ric}=\lambda g.
\end{equation}
\end{definition}
Second Ricci solitons are nothing but hyperbolic Ricci soliton with $\mu=0$. In this paper, we aimed to study some geometric aspects of these steady solitons.  
\section{Geometry of second Ricci soliton}
In this following, we establish certain findings pertaining to compact second Ricci solitons. In this paper $(M,g)$ is orientable and we integrate with respect to the volume element $\Omega_g$ of $g$.

The following formula has been proven in \cite{{6}, {7}} for any arbitrary vector field $W$ on  $(M,g)$.
\begin{equation}\label{4}
\operatorname{trace}(\mathcal{L}_{W}\mathcal{L}_{W}g) = 2 \bigl(\|\nabla W\|^{2} + \operatorname{div}(\nabla_{W}W) - \operatorname{Ric}(W,W)\bigr),
\end{equation}
where Ric denotes the Ricci curvature, ${Q}$ represents the Ricci operator defined such that $g({Q}W,Y) = \operatorname{Ric}(W,Y)$ for any smooth vector fields $W,Y$ defined on $M$, and $
\nabla$ signifies the Levi-Civita connection associated with $g$. At this point, we will initially deduce:\\

\begin{theorem}
 Let $(M^n,g,W,\lambda)$ be a second Ricci soliton:\\
 (i) When $M$ is  connected and divergence of  $\mathcal{L}_{W}\mathcal{L}_{W}g$ is equal to zero, then the scalar curvature $R$ is constant.\\
(ii) If $\mathcal{L}_{W}\mathcal{L}_{W}g$ is traceless, then the scalar curvature of $M$ is constant. Moreover, if $M$ is oriented, closed and $\int_{M} \operatorname{Ric}(W,W) \leq 0$, then $\nabla W=0$, hence $W$ is a parallel vector field. \\
\end{theorem}
\begin{proof} 
(i) Applying divergence operator on \eqref{metal}, we get
\[dR=0.\]
Since  $M$ is connected, so $R$ is constant on entire of $M$.\\
(ii) When $\operatorname{trace}(\mathcal{L}_{W}\mathcal{L}_{W}g) = 0$, tracing into the equation (\ref{metal}) it follows that
\[
R=n\lambda.
\]
Also,  we mentioned
\[
{\rm trace}(L_W L_W g)=\|\nabla W\|^{2} - \operatorname{Ric}(W,W) - \operatorname{div}(\nabla_{W}W),
\]
hence
\[
\int_{M} \|\nabla W\|^{2} = \int_{M} \operatorname{Ric}(W,W) \leq 0,
\]
and we get $\nabla W = 0$.
\end{proof}

\begin{proposition} Let  vector field $W$ be a the potential field of second Ricci soliton $(M,g,W,\lambda)$. Then
\[
\int_{M} \|\mathcal{L}_{W}\mathcal{L}_{W}g\|^{2} = \int_M(\lambda^2-2\lambda R) +\int_M \|{\rm Ric}\|^2.
\]

Moreover, if the following condition holds \\
\[\int_{M} \|\mathcal{L}_{W}\mathcal{L}_{W}g\|^{2} \leq \int_M(\lambda^2-2\lambda R),\] \\

then  $(M,g)$ is a Ricci flat manifold.
\end{proposition}

\begin{proof} We calculate the Hilbert--Schmidt norms as shown in equation (\ref{metal}), from which we deduce
\begin{align*}
\|\mathcal{L}_{W}\mathcal{L}_{W}g\|^{2} &= \langle \lambda g-{\rm Ric},\lambda g-{\rm Ric}\rangle\\
&=
\lambda^{2} \|\mathrm{Id}\|^{2} -2\lambda  \langle g, {\rm Ric} \rangle +  \|{\rm Ric}\|^{2}\\
&=\lambda^2-2\lambda R +\|{\rm Ric}\|^2.
\end{align*}
Now, considering $M$ to be closed, then by integrating we obtain
\[\int_M \|\mathcal{L}_{W}\mathcal{L}_{W}g\|^{2}=\int_M \lambda^2-2\lambda R +\|{\rm Ric}\|^2.\]
The inequality
\[\int_{M} \|\mathcal{L}_{W}\mathcal{L}_{W}g\|^{2} \leq \int_M(\lambda^2-2\lambda R),\]
provides 
\[\int_M \|{\rm Ric}\|^2=0,\]
hence $\|{\rm Ric}\|=0$ and consequently ${\rm Ric}=0$.  
\end{proof}

In the unique scenario where the metallic field is identified as a ${W}( \operatorname{Ric})$-vector field \cite{18}, we can establish the following proposition.

\begin{proposition} 
A second Ricci soliton manifold $(M^{n},g,W,\lambda)$ with ${W}(\operatorname{Ric})$-vector field that meets the conditions $\nabla W = a {Q}$, $a \in \mathbb{R}^{*}$, and\\ trace $(\mathcal{L}_{W} \operatorname{Ric}) = 0={\rm Ric}(W,W)=0$, is characterized as Ricci-flat, and $W$ is a parallel vector field ($a\neq 0$). Furthermore, the converse is also valid.
\end{proposition}
\begin{proof}
One can easily check 
\[
\operatorname{div}(W) = a R, \quad \mathcal{L}_{W}g = 2a \operatorname{Ric}, \quad \mathcal{L}_{W} \operatorname{Ric} = 2a \mathcal{L}_{W}(\operatorname{Ric}),
\]
and the equation (\ref{metal}) becomes
\[
2a \mathcal{L}_{W} \operatorname{Ric} +{\rm Ric}=\lambda g.
\]
Since ${\rm trace}(\mathcal{L}_{W} {\rm Ric}) = 0$, we get $R=n\lambda$. It follows that $R$ is a constant. Consequently, trace$(\mathcal{L}_{W} \mathcal{L}_{W} g) = 0$ and (\ref{4}) gives
\begin{equation}\label{4343}
\|\nabla W\|^{2} + \operatorname{div}(\nabla_{W} W) - \operatorname{Ric}(W,W) = 0,
\end{equation}
and by assumption, we have
\[
\operatorname{Ric}(W,W) = 0.
\]
As
\[
\|\nabla W\|^{2} = a^{2} \|{Q}\|^{2}, \quad \operatorname{div}(\nabla_W W) =a \operatorname{div}({Q}W).
\]
 By integrating \eqref{4343}, we infer
\[
a^{2} \|{Q}\|^{2}  = 0.
\]
Hence, ${Q} = 0$ and $\nabla W = 0$, so we get the conclusion.
\end{proof}

\section{Second Ricci solitoinc submanifolds}

Let ${N}$ represent a Riemannian manifold equipped with the metric $\bar{g}$, while $M$ denotes an isometrically immersed submanifold within $N$, possessing the induced metric $g$. For any smooth vector fields $W$ and $Y$ defined on $M$, along with any normal vector field $V$, the Gauss and Weingarten equations are expressed as follows:
\[\bar{
\nabla}_{W} Y =
\nabla_{W} Y + h(W,Y), \quad \bar{
\nabla}_{W} W = - A_{W} W +
\nabla^{\perp}_{W} W,\]
where $\bar{
\nabla}$ denotes the Levi-Civita connection corresponding to $\bar{g}$, $
\nabla$ signifies the Levi-Civita connection associated with $g$, $h$ represents the second fundamental form, $A_{W}$ indicates the shape operator linked to $W$, and $
\nabla^{\perp}$ is the normal connection. \\
In the following discussion, we will assume that $M$ contains a second soliton vector field, specifically the tangential component $W^{\top}$ of a concurrent vector field $W$ existing on the manifold $(\bar{M}, \bar{g})$. Consequently, it holds that $\bar{
\nabla} W= I$, where $I$ denotes the identity map. Furthermore, for any vector fields $U$ and $V$ that are tangent to $M$, we find \cite{blaga}
\begin{align*}
\nabla_U W^{\top} &= U + A_{W^{\top}} U,\\
(\mathcal{L}_{W^{\top}} g)(U, V) &= 2 \bigl( g(U, V) + g(A_{W^{\perp}} U, U) \bigr),\\
(\mathcal{L}_{W^{\top}} \mathcal{L}_{W^{\top}}g)(U, V) &= 2\Big( 2 g(U, V) + 4 g(A_{W^{\perp}} U, V) + 2 g(A_{W^{\perp}}^2 U, V)\\
&+ g\bigl( (\nabla_{W^\top} A_{W^{\perp}}) U, V \bigr)\Big).
\end{align*}

We would like to remind \cite{23, {33}} that a hypersurface is classified as a \emph{metallic shaped hypersurface} if its shape operator $A$ adheres to the equation
\[
A^2 = r A + s I,
\]
where $r$ and $s$ are real constants. It has been established that for a hypersurface situated within a space of constant curvature, if the aforementioned condition holds at a point, the hypersurface qualifies as pseudosymmetric (for further information, refer to \cite{12,13}).

At this point, we can present the subsequent findings.

\begin{proposition}
If $(M,g,W^\top, \lambda)$ represents a second soliton on the hypersurface $(M,g)$ exhibiting a parallel shape operator (specifically, $\nabla A_{W^\perp} = 0$), it is Ricci flat if and only if $(M,g)$ be a metallic shaped hypersurface with parameters $r=-2$ and $s=\dfrac{\lambda -4}{4}$.
\end{proposition}
\begin{proof} From the equation (\ref{metal}) and utilizing the above relations, we derive
\begin{align*}
(\mathcal{L}_{W^{\top}} \mathcal{L}_{W^{\top}}g)(U, V)+{\rm Ric}(U,V)&=\lambda g(U,V),\\
 g(8A_{W^{\perp}}U +4A_{W^{\perp}}^2 U, V)+{\rm Ric}(U,V)&=(\lambda -4) g(U,V),
\end{align*}
for any vector fields $U, V$ tangent to $M$, hence $M$ is Ricci flat if and only if
\[
A_{W^{\perp}}^2 =  -2 A_{W^{\perp}} +\frac{\lambda-4}{4} I,
\]
and we reach the conclusion.
\end{proof}

\begin{proposition}
Let $(M^n,g,W^\top,\lambda)$ be a second soliton with $n\geq 3$, and consider $(M, g)$ as a $W^\top$-totally umbilical submanifold characterized by the condition $A_{W^{\perp}} = f I$, where $f$ represents a smooth real-valued function defined on $M$. Then $M$ is Einstein manifold  provides that
 \[\lambda-4-8f-4f^2-2W^\top (f),\] is a real constant.

\end{proposition}
\begin{proof}
 In this case,
\[
A_{W^{\perp}}^2 U = f^2 U, \quad (\nabla_{W^{\top}} A_{W^{\perp}}) U= W^{\top}(f) U
,\]
where $U$ is arbitrary vector field, and we reach
from the equation (\ref{metal}) 
\begin{align*}
&(\mathcal{L}_{W^{\top}} \mathcal{L}_{W^{\top}}g)(U, V)+{\rm Ric}(U,V)=\lambda g(U,V)\Rightarrow\\
 &\Big( 4 g(U, V) + 8g(A_{W^{\perp}} U, V) + 4 g(A_{W^{\perp}}^2 U, V)
+ 2g\bigl( (\nabla_{W^\top} A_{W^{\perp}}) U, V \bigr)\Big)+{\rm Ric}(U,V)\\
&=\lambda g(U,V).
\end{align*}
Now, we can write
\[
{\rm Ric}(U,V)=\Big( \lambda-4-8f-4f^2-2W^\top (f)\Big)g(U,V)
\]
which leads us to our conclusion.
\end{proof}

It is important to note that a \emph{totally geodesic submanifold} refers to a submanifold characterized by a vanishing shape operator, hence, as a consequence we deduce:

\begin{proposition} Totally geodesic submanifold  $(M,g,W^{\top},\lambda)$ is Einstein.
\end{proposition}
\begin{proposition}
If $(M^n, g)$ denotes a compact minimal submanifold and $W^\top$ be second soliton vector field, then
\[
\int_M \|A_{W^{\perp}}\|^2 = \int_M n(n-1)-{\rm Ric}(W^\top,W^\top).
\]
\end{proposition}
\begin{proof}
Given that $M$ is compact, we can conclude that \cite{yano}
\[
\int_M \left( \mathrm{Ric}(W^{\top}, W^{\top}) + \frac{1}{2} \|\mathcal{L}_{W^{\top}} g\|^2 - \|\nabla W^{\top}\|^2 - (\mathrm{div}(W^{\top}))^2 \right) = 0.
\]

Direct computations give
\[
\|\mathcal{L}_{W^{\top}} g\|^2 = 4 \left( \|A_{W^{\perp}}\|^2 + 2 \mathrm{trace}(A_{W^{\perp}}) + n \right) = 4 \left( \|A_{W^{\perp}}\|^2 + n \right),
\]
\[
\|\nabla W^{\top}\|^2 = \|A_{W^{\perp}}\|^2 - 2 \mathrm{trace}(A_{W^{\perp}}) + n = \|A_{W^{\perp}}\|^2 +n,
\]
\[
(\mathrm{div}(W^{\top}))^2 = (\mathrm{trace}(A_{W^{\perp}}))^2 + 2 n \mathrm{trace}(A_{W^{\perp}}) + n^2 = n^2,
\]
given that $M$ represents a minimal submanifold. Hence, we get the result. 
\end{proof}

\begin{proposition}
Consider a minimal submanifold denoted as $(M^n, g)$ where $W^\top =
\nabla \psi$ is potential field and it is known that $\mathcal{L}_{W^{\top}} \mathcal{L}_{W^{\top}} g$ exhibits divergence-free properties. Under these conditions, we have
\[
\frac{1}{2} \Delta (\|\nabla \psi\|^2) = \|A_{\nabla \psi}\|^2 +n +{\rm Ric}(\nabla \psi,\nabla\psi).
\]
Moreover, if $M$ is closed, then
\[\int_M \|A_{\nabla \psi} \|^2 = -\int_M {\rm Ric}(\nabla\psi,\nabla\psi).\]
\end{proposition}

\begin{proof}
Applying Bochner's formula yields
\[
\frac{1}{2} \Delta (|\nabla \psi|^2) = \|\nabla W^{\top}\|^2 + W^{\top}(\mathrm{div}(W^{\top})) + \mathrm{Ric}(W^{\top}, W^{\top}),
\]
and through the calculations conducted previously, we arrive at the initial conclusion. Assuming that $M$ is closed, and by integrating this equation and taking into account that $n=\Delta (f)$, we derive
\[\int_M \|A_{\nabla \psi}\|^2 =-\int_M {\rm Ric}(W^{\top},W^{\top}),\] 
and the proof is completed.
\end{proof}
It follows that $\|\nabla \psi\|^2$ qualifies as a subharmonic function (specifically, $\Delta (\|\nabla \psi\|^2) \geq 0$) if ${\rm Ric}(\nabla\psi,\nabla\psi)\geq 0$. In this case, we have $A_{\nabla\psi}=0$ i.e, $M$ is totally geodesic.
\section{Second solitons on warped product manifolds}
In the following, we present two outcomes concerning metallic vector fields within warped product manifolds. Throughout the remainder of this section, $(M_i,g_i)$ denotes semi-Riemannian manifolds for $i=1,2$, while $M=M_1\times_f M_2$ represents a warped product semi-Riemannian manifold characterized by the metric tensor $g=g_1+f^2g_2$ along with the warping function $f:M_1\to \mathbb{R}$. A smooth vector field $U$ defined on $M$ can be expressed as $U=U_1+U_2$, where each $U_i$ symbolizes smooth vector fields that are tangent to their respective manifolds $M_i$. Let $\dim M_2=k$, we know
\begin{align}
{\rm Ric}(U_1,V_1)&={\rm Ric}_1(U_1,V_1)-\dfrac{k}{f}{\rm Hess}_1 f(U_1,V_1),\\
{\rm Ric}(U_1,V_2)&=0,\\
{\rm Ric}(U_2,V_2)&={\rm Ric}_2(U_2,V_2)-\Big(\dfrac{\Delta f}{f}+(k-1)\dfrac{\langle \nabla f,\nabla f \rangle}{f^2}\Big)g_2(U_2,V_2).
\end{align}
Remind that a semi-Riemannian manifold $(M,g)$ is said to be an $h-$almost Ricci soliton if there exist real smooth functions $h,\lambda$ on $M$ and a vector field $X\in \mathcal{X}(M)$, such that
\[h\mathcal{L}_X g+{\rm Ric}=\lambda g.\]
\begin{proposition}
Let vector field  $W=W_1+W_2$ be 2-Killing  on $M_1\times_f M_2$, then:\\
 (i) $W_1$ is potential vector field for the second soliton $(M_1,g_1,\lambda)$ if and only if $(M_1,g_1)$ be an Einstein manifold
\[{\rm Ric}_1=\lambda g_1;\]
 \\
 (ii) $(M_2,g_2, W_2,\lambda)$ is second soliton if and only if 
\[{\rm Ric}_2-4W_1(\ln f)\mathcal{L}_{W_2}g_2=(\lambda+\dfrac{W_1(W_1(f^2))}{f^2})g_2,\]
 provided that $(M_2,g_2)$ is an $h$-almost Ricci soliton \cite{almost} with $h=-4W_1(\ln f)$.
\end{proposition}
\begin{proof}
From \cite{7}, for all $U=U_1+U_2$ and $V=V_1+V_2$ we have
\begin{align*}
(\mathcal{L}_W \mathcal{L}_W g)(U,V)&= (\mathcal{L}_{W_1} \mathcal{L}_{W_1} g_1)(U_1,V_1)+f^2(\mathcal{L}_{W_2} \mathcal{L}_{W_2} g_2)(U_2,V_2)\\
+&2W_1(f^2)(\mathcal{L}_{W_2} g_2)(U_2,V_2)+W_1(W_1(f^2))g_2(U_2,V_2).
\end{align*}
Since, $W$ is 2-Killing vector filed, then 
\begin{align}
 (\mathcal{L}_{W_1} \mathcal{L}_{W_1} g_1)&=0,\nonumber\\
 f^2(\mathcal{L}_{W_2} \mathcal{L}_{W_2} g_2)
+2W_1(f^2)(\mathcal{L}_{W_2} g_2)+W_1 W_1(f^2)g_2&=0.\label{434}
\end{align}
If $(M_1,g_1,W_1,\lambda)$ be a second Ricci soliton, then
\[\mathcal{L}_{W_1} \mathcal{L}_{W_1} g_1+{\rm Ric}_1=\lambda g_1,\]
hence, we have ${\rm Ric}_1=\lambda g_1$. \\
Similarly, if $(M_2,g_2,W_2,\lambda)$ is a second Ricci soliton, then 
\[\mathcal{L}_{W_2} \mathcal{L}_{W_2} g_2+{\rm Ric}_2=\lambda g_2.\]
Now, this formula beside (\ref{434}), gives\\
\[
-\dfrac{2W_1(f^2)}{f^2}(\mathcal{L}_{W_2} g_2)-\dfrac{W_1 W_1(f^2)}{f^2}g_2+{\rm Ric}_2=\lambda g_2,
\]
so,
\[{\rm Ric}_2-\dfrac{2}{f^2}W_1(f^2)\mathcal{L}_{W_2}g_2=(\lambda+\dfrac{W_1(W_1(f^2))}{f^2})g_2.\]
\end{proof}
One can easily extend the above result for multiply warped product manifolds $M_1\times_{f_1} M_2\times \cdots \times_{f_n}M_n$.
\begin{proposition}
When manifold $M_1\times_f M_2$ admits a
 second Ricci soliton vector field $W=W_1 +W_2$ with scalar $\lambda$, then:\\
 (i)  $(M_2,g_2)$ satisfies
 \[\mathcal{L}_{W_1}\mathcal{L}_{W_1} g_1+{\rm Ric}_1-\dfrac{k}{f}{\rm Hess}_1 f=\lambda g_1.\]
 Moreover, $(M_1,g_1, W_1,\lambda)$ is a second Ricci soliton if and only if
 \[{\rm Hess}_1 f=0.\]
 (ii)  $(M_2,g_2)$ satisfies
 \[\mathcal{L}_{W_2} \mathcal{L}_{W_2} g_2+\dfrac{{\rm Ric}_2}{f^2}+4W_1(\ln f)\mathcal{L}_{W_2} g_2=\dfrac{1}{f^2}\big(\lambda f^2+\dfrac{\Delta f}{f}+(k-1)\dfrac{\langle \nabla f, \nabla f \rangle}{f^2}\big)g_2.\]
\end{proposition}
\begin{proof}
We begin with
\[(\mathcal{L}_W\mathcal{L}_W g)(U,V)+{\rm Ric}(U,V)=\lambda g(U,V),\]
to get
\begin{align*}
&(\mathcal{L}_{W_1} \mathcal{L}_{W_1} g_1)(U_1,V_1)+f^2(\mathcal{L}_{W_2} \mathcal{L}_{W_2} g_2)(U_2,V_2)+2W_1(f^2)(\mathcal{L}_{W_2} g_2)(U_2,V_2)\\&+W_1(W_1(f^2))g_2(U_2,V_2)+{\rm Ric}_1(U_1,V_1)-\dfrac{k}{f}{\rm Hess}_1 f(U_1,V_1)\\
&+{\rm Ric}_2(U_2,V_2)-\Big(\dfrac{\Delta f}{f}+(k-1)\dfrac{\langle \nabla f,\nabla f \rangle}{f^2}\Big)g_2(U_2,V_2)\\
&=\lambda g_1(U_1,V_1)+\lambda f^2 g_2(U_2,V_2),
\end{align*}
for all vector fields $U_i, V_i$ tangent to $M_i$. Hence, we have
\begin{align*}
\mathcal{L}_{W_1}\mathcal{L}_{W_1} g_1+{\rm Ric}_1-\dfrac{k}{f}{\rm Hess}_1 f=\lambda g_1,\\
\mathcal{L}_{W_2} \mathcal{L}_{W_2} g_2+\dfrac{{\rm Ric}_2}{f^2}+4W_1(\ln f)\mathcal{L}_{W_2} g_2&=\dfrac{1}{f^2}\big(\lambda f^2+\dfrac{\Delta f}{f}+(k-1)\dfrac{\langle \nabla f, \nabla f \rangle}{f^2}\big)g_2.
\end{align*}
The above equation indicates $(M_1,g_1)$ is a second Ricci soliton if and only if ${\rm Hess}_1 f=0$.
\end{proof}
{\bf Remark}: Let the warping function $f$ be constant. When  $M_1\times M_2$ admits a second soliton vector field $W=W_1+W_2$, then $W_i$ are second solitonic fields over $M_i$.

\section*{Conclusion}
In this paper we introduced the idea of second Ricci solitons. A second Ricci soliton serves as a steady hyperbolic Ricci soliton. Geometry of closed and compact second Ricci soliton manifolds is considered.  We also, investigated immersed submanifolds as second  Ricci solitons. Makeing use of the Gauss–Weingarten formulas, some results (Propositions 3.2–3.4) obtained which in line with existing literature. Finally, we considered this structure on warped product manifolds. Study of existence of this new structure on model spaces may be the matter of other researches.

\end{document}